\documentclass[12pt,english,reqno]{amsart}
\usepackage{charter}
\usepackage[T1]{fontenc}
\usepackage[latin9]{inputenc}
\usepackage{geometry}
\usepackage{color}
\usepackage{babel}
\usepackage{prettyref}
\usepackage{enumitem}
\usepackage{amstext}
\usepackage{amsthm}
\usepackage{amssymb}
\usepackage{setspace}
\usepackage[numbers]{natbib}
\usepackage[unicode=true,
 bookmarks=false,
 breaklinks=false,pdfborder={0 0 1},backref=page,colorlinks=true]
 {hyperref}
\hypersetup{pdftitle={Amenability of locally compact quantum groups and their unitary co-representations},
 pdfauthor={Chi-Keung Ng and Ami Viselter},
 pdfstartview={XYZ null null 1.25},citecolor=OliveGreen,linkcolor=RoyalBlue,urlcolor=blue}

\makeatletter
\numberwithin{equation}{section}
\theoremstyle{plain}
\newtheorem{thm}{\protect\theoremname}[section]
  \theoremstyle{definition}
  \newtheorem{defn}[thm]{\protect\definitionname}
  \theoremstyle{plain}
  \newtheorem{lem}[thm]{\protect\lemmaname}
  \theoremstyle{remark}
  \newtheorem{rem}[thm]{\protect\remarkname}
  \theoremstyle{plain}
  \newtheorem{prop}[thm]{\protect\propositionname}
  \theoremstyle{plain}
  \newtheorem{cor}[thm]{\protect\corollaryname}
  \theoremstyle{remark}
  \newtheorem*{acknowledgement*}{\protect\acknowledgementname}

\usepackage{dsfont}
\usepackage{graphicx}
\usepackage{mathtools} % for \prescript
\usepackage[all]{xy}
\usepackage[usenames,dvipsnames]{xcolor}
\usepackage{upgreek}

\newrefformat{eq}{\eqref{#1}}
\newrefformat{def}{Definition \ref{#1}}
\newrefformat{lem}{Lemma \ref{#1}}
\newrefformat{pro}{Proposition \ref{#1}}
\newrefformat{prop}{Proposition \ref{#1}}
\newrefformat{cor}{Corollary \ref{#1}}
\newrefformat{thm}{Theorem \ref{#1}}
\newrefformat{exa}{Example \ref{#1}}
\newrefformat{rem}{Remark \ref{#1}}
\newrefformat{sec}{Section \ref{#1}}
\newrefformat{sub}{Subsection \ref{#1}}
\newrefformat{conj}{Conjecture \ref{#1}}
\newrefformat{desc}{\ref{#1}}
\newrefformat{ques}{Question \ref{#1}}

\makeatother

  \providecommand{\acknowledgementname}{Acknowledgement}
  \providecommand{\corollaryname}{Corollary}
  \providecommand{\definitionname}{Definition}
  \providecommand{\lemmaname}{Lemma}
  \providecommand{\propositionname}{Proposition}
  \providecommand{\remarkname}{Remark}
\providecommand{\theoremname}{Theorem}

\begin{document}
\global\long\def\e{\varepsilon}
\global\long\def\N{\mathbb{N}}
\global\long\def\Z{\mathbb{Z}}
\global\long\def\Q{\mathbb{Q}}
\global\long\def\RR{\mathbb{R}}
\global\long\def\CC{\mathbb{C}}
\global\long\def\G{\mathbb{G}}
\global\long\def\HH{\mathbb{H}}

\global\long\def\H{\mathcal{H}}
\global\long\def\J{\mathcal{J}}
\global\long\def\K{\mathcal{K}}
\global\long\def\a{\alpha}
\global\long\def\be{\beta}
\global\long\def\om{\omega}
\global\long\def\z{\zeta}
\global\long\def\gnsmap{\upeta}
\global\long\def\Aa{\mathcal{A}}
\global\long\def\JJ{\mathscr{J}}
\global\long\def\M{\mathrm{M}}

\global\long\def\Ree{\operatorname{Re}}
\global\long\def\Img{\operatorname{Im}}
\global\long\def\linspan{\operatorname{span}}
\global\long\def\slim{\operatorname*{s-lim}}
\global\long\def\clinspan{\operatorname{\overline{span}}}
\global\long\def\co{\operatorname{co}}
\global\long\def\inv{\operatorname{Inv}}
\global\long\def\CAP{\mathrm{CAP}}
\global\long\def\presb#1#2{\prescript{}{#1}{#2}}

\global\long\def\tensor{\otimes}
\global\long\def\tensormin{\mathbin{\tensor_{\mathrm{min}}}}
\global\long\def\tensorn{\mathbin{\overline{\otimes}}}

\global\long\def\A{\forall}

\global\long\def\i{\mathrm{id}}

\global\long\def\one{\mathds{1}}
\global\long\def\tr{\mathrm{tr}}
\global\long\def\Ww{\mathds{W}}
\global\long\def\wW{\text{\reflectbox{\ensuremath{\Ww}}}\:\!}
\global\long\def\op{\mathrm{op}}
\global\long\def\WW{{\mathds{V}\!\!\text{\reflectbox{\ensuremath{\mathds{V}}}}}}

\global\long\def\Linfty#1{L^{\infty}(#1)}
\global\long\def\Lone#1{L^{1}(#1)}
\global\long\def\LoneSharp#1{L_{\sharp}^{1}(#1)}
\global\long\def\Ltwo#1{L^{2}(#1)}
\global\long\def\Cz#1{C_{0}(#1)}
\global\long\def\CzU#1{C_{0}^{u}(#1)}
\global\long\def\CUcomp#1{C^{u}(#1)}
\global\long\def\Ccomp#1{C(#1)}
\global\long\def\Cb#1{C_{b}(#1)}
 \global\long\def\Cc#1{C_{c}(#1)}
\global\long\def\Cr#1{C_{r}^{*}(#1)}
\global\long\def\linfty#1{\ell^{\infty}(#1)}
\global\long\def\lone#1{\ell^{1}(#1)}
\global\long\def\ltwo#1{\ell^{2}(#1)}
\global\long\def\cz#1{\mathrm{c}_{0}(#1)}
\global\long\def\Pol#1{\mathrm{Pol}(#1)}
\global\long\def\Ltwozero#1{L_{0}^{2}(#1)}
\global\long\def\Irred#1{\mathrm{Irr}(#1)}

\global\long\def\Ad#1{\mathrm{Ad}(#1)}
\global\long\def\VN#1{\mathrm{VN}(#1)}
\global\long\def\d{~\mathrm{d}}
\global\long\def\t{\mathrm{t}}
\global\long\def\tp{\mathbin{\xymatrix{*+<.7ex>[o][F-]{\scriptstyle \top}}
 } }
\global\long\def\tpsmall{\mathbin{\xymatrix{*+<.5ex>[o][F-]{\scriptscriptstyle \top}}
 } }
\global\long\def\tpr{\mathbin{\xymatrix{*+<.7ex>[o][F-]{\scriptstyle \bot}}
 } }
\global\long\def\tprsmall{\mathbin{\xymatrix{*+<.5ex>[o][F-]{\scriptscriptstyle \bot}}
 } }

\title{Amenability of locally compact quantum groups and their unitary co-representations}

\author{Chi-Keung Ng}

\address{Chern Institute of Mathematics and LPMC, Nankai University, Tianjin
300071, China}

\email{ckng@nankai.edu.cn}

\author{Ami Viselter}

\address{Department of Mathematics, University of Haifa, 31905 Haifa, Israel}

\email{aviselter@staff.haifa.ac.il}
\begin{abstract}
We prove that amenability of a unitary co-representation $U$ of a
locally compact quantum group passes to unitary co-representations
that weakly contain $U$. This generalizes a result of Bekka, and
answers affirmatively a question of B{\'e}dos, Conti and Tuset. As
a corollary, we extend to locally compact quantum groups a result
of the first-named author, which characterizes amenability of a locally
compact group $G$ by nuclearity of the reduced group $C^{*}$-algebra
$\Cr G$ and an additional condition.
\end{abstract}

\maketitle

\section*{Introduction}

A well-known result of Lance says that if a locally compact group
$G$ is amenable then its reduced group $C^{*}$-algebra $\Cr G$
is nuclear, and that the converse holds when $G$ is discrete \citep[Proposition 4.1 and Theorem 4.2]{Lance__nuclear_C_alg}
(but not generally; see Connes \citep[Corollary 7]{Connes__classification_of_inj_factors}).
It is thus interesting to look for a condition whose combination with
nuclearity of $\Cr G$ is equivalent to amenability of $G$ for an
arbitrary locally compact group $G$. This problem was solved recently
by the first-named author \citep[Theorem 8]{Ng__strictly_amen_rep},
who proved that $G$ is amenable if and only if $\Cr G$ is nuclear
and possesses a tracial state. In \prettyref{sec:amen_LCQGs} we extend
this theorem to locally compact \emph{quantum} groups in the sense
of Kustermans and Vaes, replacing traciality with a suitable ``noncommutative''
condition. Our result bears some resemblance to recent characterizations
of amenability in terms of various notions of injectivity, which have
proven to admit many applications \citep{Soltan_Viselter__note_amenability_LCQGs,Crann_Neufang__amn_cov_inj,Crann__amn_cov_inj_2}.
However, it focuses on the dual reduced $C^{*}$-algebra rather than
the dual $L^{\infty}$-algebra.

A key tool used in our proof is that of amenability of (unitary) co-representations
of locally compact quantum groups. This notion was introduced for
groups in the fundamental work of Bekka \citep{Bekka__amen_unit_rep}.
It is related to group amenability by the fact that a locally compact
group $G$ is amenable if and only if every representation of $G$
is amenable, if and only if the left regular representation of $G$
is amenable. Another useful result of \citep{Bekka__amen_unit_rep}
asserts that if $\pi_{1},\pi_{2}$ are representations of $G$ such
that $\pi_{1}$ is amenable and is weakly contained in $\pi_{2}$,
then $\pi_{2}$ is also amenable. B{\'e}dos, Conti and Tuset \citep{Bedos_Conti_Tuset__amen_co_amen_alg_QGs_coreps}
and B{\'e}dos and Tuset \citep{Bedos_Tuset_2003} introduced amenability
of co-representations of locally compact quantum groups, generalizing
Bekka's notion. One question they left open was whether amenability
was well-behaved with respect to weak containment as proved for groups
by Bekka. We provide an affirmative answer to this question in \prettyref{sec:amen_co_reps}.
It is then employed to establish the main result of \prettyref{sec:amen_LCQGs}.
We remark that amenability of co-representations of Kac algebras was
called ``weak Bekka amenability'' in \citep{Ng__amen_rep_Reiter}
up to a difference in the convention of what a co-representation is.

\section{Preliminaries}

For a (complex) Hilbert space $\H$ we denote by $B(\H)$, respectively
$K(\H)$, the $C^{*}$-algebra of all bounded, respectively compact,
operators on $\H$. For $\z,\eta\in\H$ we define $\om_{\z,\eta}\in B(\H)_{*}$
by $\om_{\z,\eta}(x):=\left\langle x\z,\eta\right\rangle $ ($x\in B(\H)$)
and $\om_{\z}:=\om_{\z,\z}$. Representations of $C^{*}$-algebras
are assumed to be nondegenerate. For a $C^{*}$-algebra $A$, we write
$\i$ for the identity map on $A$ and $\one$ for the unit of $A$,
if exists. We denote by $\M(A)$ the multiplier algebra of $A$. For
details on multiplier algebras, the strict topology and related topics,
consult \citep{Lance}. The symbols $\tensormin$ and $\tensorn$
stand for the minimal tensor product of $C^{*}$-algebras and the
normal spatial tensor product of von Neumann algebras, respectively.
We will use terminology and results from operator space theory; see
\citep{Effros_Ruan__book} as a general reference.

A \emph{locally compact quantum group} (abbreviated LCQG) is a pair
$\G=\left(\Linfty{\G},\Delta\right)$, where $\Linfty{\G}$ is a von
Neumann algebra and $\Delta$ is a co-multiplication, namely a normal
unital $*$-homomorphism from $\Linfty{\G}$ to $\Linfty{\G}\tensorn\Linfty{\G}$
that is co-associative: $(\Delta\tensor\i)\Delta=(\i\tensor\Delta)\Delta$,
admitting a left-invariant weight and a right-invariant weight \citep{Kustermans_Vaes__LCQG_C_star,Kustermans_Vaes__LCQG_von_Neumann,Van_Daele__LCQGs}.
The precise definition of left/right invariance will not be needed
here explicitly, and so we refer the reader to the above references
for details, as well as for the following facts. Each LCQG $\G$ has
a \emph{dual} LCQG, denoted by $\hat{\G}$. The von Neumann algebras
$\Linfty{\G}$ and $\Linfty{\hat{\G}}$ act standardly on the same
Hilbert space $\Ltwo{\G}$. A very important object is the left regular
co-representation of $\G$, which is a multiplicative unitary $W\in\Linfty{\G}\tensorn\Linfty{\hat{\G}}$
that satisfies $\Delta(x)=W^{*}(\one\tensor x)W$ for all $x\in\Linfty{\G}$.
The von Neumann algebra $\Linfty{\G}$ has a canonical weakly dense
$C^{*}$-subalgebra $\Cz{\G}$, and we have $W\in\M(\Cz{\G}\tensormin\Cz{\hat{\G}})$.
We write $\Lone{\G}$ for the predual $\Linfty{\G}_{*}$.

A (unitary) \emph{co-representation} of a LCQG $\G$ on a Hilbert
space $\H$ is a unitary operator $U\in\M(\Cz{\G}\tensormin K(\H))$
that satisfies $(\Delta\tensor\i)(U)=U_{13}U_{23}$. The \emph{universal}
picture of $\G$ involves another $C^{*}$-algebra, $\CzU{\G}$. There
exists a co-representation $\wW\in\M(\Cz{\G}\tensormin\CzU{\hat{\G}})$
of $\G$ with the property that there is a bijection between co-representations
$U$ of $\G$ and representations $\pi$ of $\CzU{\hat{\G}}$ given
by $U=(\i\tensor\pi)(\wW)$. For this, see \citep{Kustermans__LCQG_universal}.

The simplest examples of LCQGs are given by locally compact groups
$G$. The associated $L^{\infty}$, $C_{0}$ and $C_{0}^{u}$ algebras
are just $\Linfty G$, $\Cz G$ and $\Cz G$, respectively, and $\Delta$
maps a function $f\in\Linfty G$ to the function $\Delta(f)\in\Linfty G\tensorn\Linfty G\cong\Linfty{G\times G}$
given by $(t,s)\mapsto f(ts)$, $t,s\in G$. The dual of $G$ is the
LCQG $\hat{G}$ whose associated algebras $\Linfty{\hat{G}}$, $\Cz{\hat{G}}$
and $\CzU{\hat{G}}$ are the (left) group von Neumann algebra $\VN G$,
the reduced group $C^{*}$-algebra $\Cr G$ and the full group $C^{*}$-algebra
$C^{*}(G)$ of $G$, respectively, and the co-multiplication of $\hat{G}$
maps $\lambda_{g}$ to $\lambda_{g}\tensor\lambda_{g}$ for every
$g$, where $\left(\lambda_{g}\right)_{g\in G}$ is the left regular
representation of $G$.

A LCQG $\G$ is called \emph{compact} if $\Cz{\G}$ is unital. This
is equivalent to the left- and right-invariant weights being equal
and finite. A LCQG is called \emph{discrete} if its dual is compact.
See \citep{Woronowicz__symetries_quantiques,Effros_Ruan__discrete_QGs,Van_Daele__discrete_CQs},
and also \citep{Runde__charac_compact_discr_QG} for the equivalence
of different characterizations.

\section{\label{sec:amen_co_reps}Amenability of co-representations and weak
containment}

In this section we extend an important result of Bekka \citep{Bekka__amen_unit_rep}
to LCQG co-representations. We begin with two definitions from \citep{Ng__amen_rep_Reiter,Bedos_Conti_Tuset__amen_co_amen_alg_QGs_coreps,Bedos_Tuset_2003}.
\begin{defn}
A co-representation $U$ of a LCQG $\G$ on a Hilbert space $\H$
is called \emph{left amenable}, respectively \emph{right amenable},
if there exists a state $m$ of $B(\H)$ such that $m\left((\om\tensor\i)\left(U^{*}(\one\tensor x)U\right)\right)=\om(\one)m(x)$,
respectively $m\left((\om\tensor\i)\left(U(\one\tensor x)U^{*}\right)\right)=\om(\one)m(x)$,
for every $x\in B(\H)$ and $\om\in\Lone{\G}$. The state $m$ is
said to be a left-invariant, respectively right-invariant, mean of
$U$.
\end{defn}
\begin{defn}
For $i=1,2$, let $U_{i}$ be a co-representation of a LCQG $\G$
on a Hilbert space $\H_{i}$, and write $\pi_{i}$ for the associated
representation of $\CzU{\hat{\G}}$ on $\H_{i}$. We say that $U_{1}$
is \emph{weakly contained} in $U_{2}$ if $\pi_{1}$ is weakly contained
in $\pi_{2}$ \citep[Section 3.4]{Dixmier__C_star_English}, that
is, $\ker\pi_{2}\subseteq\ker\pi_{1}$.
\end{defn}
The following result is a generalization of \citep[Corollary 5.3]{Bekka__amen_unit_rep}.
It was proved for strong instead of weak containment in \citep[Proposition 7.14]{Bedos_Conti_Tuset__amen_co_amen_alg_QGs_coreps},
and for discrete quantum groups in \citep[Corollary 9.7]{Bedos_Conti_Tuset__amen_co_amen_alg_QGs_coreps}.
It is mentioned in \citep[p.~49]{Bedos_Conti_Tuset__amen_co_amen_alg_QGs_coreps}
that the general case is open and of interest. We will use it below
to establish \prettyref{thm:amen_nuc}.
\begin{thm}
\label{thm:amen_weak_contain}Let $U_{1},U_{2}$ be co-representations
of a LCQG $\G$. Suppose that $U_{1}$ is weakly contained in $U_{2}$.
If $U_{1}$ is left (respectively, right) amenable, then so is $U_{2}$.
\end{thm}
We require the next lemma, which follows as a particular case from
\citep{Neufang__amp_cb_slice_maps}. For the reader's convenience,
we give its short proof.
\begin{lem}
\label{lem:id_tensor_Phi}Let $A,B,C$ be von Neumann algebras and
$\Phi:B\to C$ a completely bounded map. Then there exists a unique
completely bounded linear map $\i\tensor\Phi:A\tensorn B\to A\tensorn C$
such that 
\begin{equation}
(\om\tensor\i)((\i\tensor\Phi)(X))=\Phi((\om\tensor\i)(X))\qquad(\forall X\in A\tensorn B,\om\in A_{*}).\label{eq:id_tensor_Phi}
\end{equation}
It satisfies $\left\Vert \i\tensor\Phi\right\Vert _{cb}=\left\Vert \Phi\right\Vert _{cb}$.
\end{lem}
\begin{proof}
Uniqueness is clear. The map $\mathcal{CB}(A_{*},B)\to\mathcal{CB}(A_{*},C)$
given by $T\mapsto\Phi\circ T$ for $T\in\mathcal{CB}(A_{*},B)$ is
evidently well defined and with $cb$-norm at most $\left\Vert \Phi\right\Vert _{cb}$.
Using the natural completely isometric identifications $A\tensorn B\cong\mathcal{CB}(A_{*},B)$
and $A\tensorn C\cong\mathcal{CB}(A_{*},C)$ as operator spaces \citep[Theorem 7.2.4 and Proposition 7.1.2]{Effros_Ruan__book}
we get a linear map $\i\tensor\Phi:A\tensorn B\to A\tensorn C$ with
$cb$-norm at most $\left\Vert \Phi\right\Vert _{cb}$ that satisfies
\prettyref{eq:id_tensor_Phi}. Since $(\i\tensor\Phi)(a\tensor b)=a\tensor\Phi(b)$
for all $a\in A$ and $b\in B$, this implies that $\left\Vert \i\tensor\Phi\right\Vert _{cb}=\left\Vert \Phi\right\Vert _{cb}$.
\end{proof}
Recall that a unital linear map between operator systems is completely
positive if and only if it is completely contractive.

\begin{proof}[Proof of \prettyref{thm:amen_weak_contain}]
Let $U_{i}$ be a co-representation of $\G$ on a Hilbert space $\H_{i}$,
and write $\pi_{i}$ for the associated representation of $\CzU{\hat{\G}}$
on $\H_{i}$ ($i=1,2$). By assumption, there exists a $*$-homomorphism
$\pi:\Img\pi_{2}\to\Img\pi_{1}$ given by $\pi\circ\pi_{2}=\pi_{1}$.
Denote by $\overline{\pi}$ the (unique) extension of $\pi$ to a
unital $*$-homomorphism $\M(\Img\pi_{2})\to\M(\Img\pi_{1})\subseteq B(\H_{1})$
(actually, the extension of $\pi$ to the trivial unitization of $\Img\pi_{2}$
would suffice). Viewing $\overline{\pi}$ as a representation of $\M(\Img\pi_{2})$
on $\H_{1}$, we extend it to a unital completely positive map $\Phi:B(\H_{2})\to B(\H_{1})$
by Arveson's extension theorem. Consider now the map $\i\tensor\Phi:\Linfty{\G}\tensorn B(\H_{2})\to\Linfty{\G}\tensorn B(\H_{1})$
given by \prettyref{lem:id_tensor_Phi}, which is unital and completely
positive as $\Phi$ is. The unitaries $U_{i}=(\i\tensor\pi_{i})(\wW)$,
$i=1,2$, satisfy $(\i\tensor\Phi)(U_{2})=U_{1}$, because for every
$\om\in\Lone{\G}$, 
\[
(\om\tensor\i)\left((\i\tensor\Phi)(U_{2})\right)=\Phi((\om\tensor\i)(U_{2}))=(\Phi\circ\pi_{2})((\om\tensor\i)(\wW))=\pi_{1}((\om\tensor\i)(\wW))=(\om\tensor\i)(U_{1}).
\]
Hence $U_{2}$ belongs to the multiplicative domain of $\i\tensor\Phi$
\citep[Theorem 3.18]{Paulsen__book_CB_maps_and_oper_alg}. As a result,
for every $x\in B(\H_{2})$, 
\[
(\i\tensor\Phi)\left(U_{2}^{*}(\one\tensor x)U_{2}\right)=(\i\tensor\Phi)(U_{2}^{*})(\one\tensor\Phi(x))(\i\tensor\Phi)(U_{2})=U_{1}^{*}(\one\tensor\Phi(x))U_{1}.
\]
If now $m_{U_{1}}$ is a left-invariant mean for $U_{1}$, then $m_{U_{1}}\circ\Phi$
is a left-invariant mean for $U_{2}$, because for every $x\in B(\H_{2})$
and $\om\in\Lone{\G}$,
\[
\begin{split}(m_{U_{1}}\circ\Phi)\left[(\om\tensor\i)\left(U_{2}^{*}(\one\tensor x)U_{2}\right)\right] & =m_{U_{1}}(\om\tensor\i)\left[(\i\tensor\Phi)\left(U_{2}^{*}(\one\tensor x)U_{2}\right)\right]\\
 & =m_{U_{1}}(\om\tensor\i)\left(U_{1}^{*}(\one\tensor\Phi(x))U_{1}\right)=\om(\one)(m_{U_{1}}\circ\Phi)(x).
\end{split}
\]
The proof for right amenability is similar.
\end{proof}

\section{\label{sec:amen_LCQGs}Amenability of locally compact quantum groups}

This section is devoted to \prettyref{thm:amen_nuc} below, which
generalizes the main result of \citep{Ng__strictly_amen_rep}. It
provides a condition that sits between amenability of a LCQG $\G$
and co-amenability of its dual. When $\G$ is discrete (or a group),
all three conditions are equivalent.
\begin{defn}[\citep{Desmedt_Quaegebeur_Vaes,Bedos_Tuset_2003}]
Let $\G$ be a LCQG.
\begin{enumerate}
\item We say that $\G$ is \emph{amenable} if it has a left-invariant mean,
namely a state $m$ of $\Linfty{\G}$ that satisfies 
\[
m\left((\om\tensor\i)\Delta(x)\right)=\om(\one)m(x)\qquad(\forall x\in\Linfty{\G},\om\in\Lone{\G}).
\]
\item We say that $\G$ is \emph{co-amenable} if there exists a state $\epsilon$
of $\Cz{\G}$ that satisfies $(\epsilon\tensor\i)(W)=\one$.
\end{enumerate}
\end{defn}
A locally compact group $G$ is amenable if and only if it is amenable
in the above sense when viewed as a LCQG. For every LCQG $\G$, co-amenability
of $\hat{\G}$ implies amenability of $\G$ \citep[Theorem 3.2]{Bedos_Tuset_2003}.
The converse holds when $\G$ is a locally compact group (by Leptin's
theorem) and when $\G$ is discrete (see for instance \citep{Tomatsu__amenable_discrete}).
Whether it is true in general is arguably the most important open
question in LCQG amenability theory.
\begin{thm}
\label{thm:amen_nuc}Let $\G$ be a LCQG. 
\begin{enumerate}
\item Consider the following conditions:
\begin{enumerate}[label=\textup{(\alph*)}]
\item \label{enu:amen_nuc_1_a}$\hat{\G}$ is co-amenable;
\item \label{enu:amen_nuc_1_b}$\Cz{\hat{\G}}$ is nuclear, and there exists
a state $\rho$ of $\Cz{\hat{\G}}$ that is invariant under the left
$C^{*}$-algebraic action of $\G$ on $\Cz{\hat{\G}}$, i.e., 
\begin{equation}
(\i\tensor\rho)\left(W^{*}(\one\tensor x)W\right)=\rho(x)\one\qquad(\forall x\in\Cz{\hat{\G}});\label{eq:amen_nuc}
\end{equation}
\item \label{enu:amen_nuc_1_c}$\G$ is amenable.
\end{enumerate}
Then \prettyref{enu:amen_nuc_1_a}$\implies$\prettyref{enu:amen_nuc_1_b}$\implies$\prettyref{enu:amen_nuc_1_c}. 
\item \label{enu:amen_nuc_2}Moreover, if $\G$ has trivial scaling group
(for instance, if $\G$ is a Kac algebra), then \prettyref{enu:amen_nuc_1_a}$\implies$\prettyref{enu:amen_nuc_2_trivial_tau}$\implies$\prettyref{enu:amen_nuc_1_b},
where
\begin{enumerate}[label=\textup{(b')}]
\item \label{enu:amen_nuc_2_trivial_tau}$\Cz{\hat{\G}}$ is nuclear and
admits a tracial state.
\end{enumerate}
\end{enumerate}
\end{thm}
\begin{rem}
Observe that in contrast to the specific case of (locally compact)
groups, the second half of condition \prettyref{enu:amen_nuc_1_b}
is not intrinsic to the $C^{*}$-algebra $\Cz{\hat{\G}}$. When $\G$
is a locally compact group, \prettyref{eq:amen_nuc} is equivalent
to $\rho$ being tracial. When $\G$ is a discrete quantum group,
the Haar state $\rho$ of $\hat{\G}$ satisfies \prettyref{eq:amen_nuc}
if and only if $\G$ is a Kac algebra by \citep[Corollary 3.9 and its proof]{Izumi__non_comm_Poisson}
(note the difference in the conventions). Also, one cannot deduce
from \prettyref{thm:amen_nuc} that for discrete $\G$, nuclearity
of $\Cz{\hat{\G}}$ implies amenability of $\G$\textemdash whether
this is true remains an open question.
\end{rem}
Recall that the \emph{antipode} of $\G$ is a generally unbounded
operator $\kappa$ over $\Linfty{\G}$ that satisfies 
\begin{equation}
(\i\tensor\rho)(W)\in D(\kappa)\text{ and }\kappa((\i\tensor\rho)(W))=(\i\tensor\rho)(W^{*})\qquad(\forall\rho\in\Lone{\hat{\G}}).\label{eq:antipode_W}
\end{equation}
The antipode has a ``polar decomposition'' $\kappa=R\circ\tau_{-i/2}$,
where $R$ is the unitary antipode, which is an anti-automorphism
of $\Linfty{\G}$, and $\tau$ is the scaling group. Thus, when $\tau$
is trivial, $\kappa=R$. 
\begin{lem}
\label{lem:antipode_W_extended}Let $\G$ be a LCQG with trivial scaling
group. For every $\theta\in\Cz{\hat{\G}}^{*}$,
\begin{equation}
\kappa((\i\tensor\theta)(W))=(\i\tensor\theta)(W^{*}).\label{eq:antipode_W_extended}
\end{equation}
\end{lem}
\begin{proof}
Fix $\theta\in\Cz{\hat{\G}}^{*}$. Let $\tilde{\theta}\in\Linfty{\hat{\G}}^{*}$
be a Hahn\textendash Banach extension of $\theta$ and $\left(\theta_{i}\right)_{i\in\mathfrak{I}}$
be a net in $\Lone{\hat{\G}}$ that converges to $\tilde{\theta}$
in the $\sigma(\Linfty{\hat{\G}}^{*},\Linfty{\hat{\G}})$-topology.
For every $\om\in\Lone{\G}$, we have
\begin{equation}
\theta_{i}\left(((\om\circ\kappa)\tensor\i)(W)-(\om\tensor\i)(W^{*})\right)\xrightarrow[i\in\mathfrak{I}]{}\theta\left(((\om\circ\kappa)\tensor\i)(W)-(\om\tensor\i)(W^{*})\right).\label{eq:antipode_W_extended_1}
\end{equation}
But by \prettyref{eq:antipode_W}, the left-hand side of \prettyref{eq:antipode_W_extended_1}
is zero for all $i\in\mathfrak{I}$. We conclude that $(\om\circ\kappa)\left((\i\tensor\theta)(W)\right)=\om\left((\i\tensor\theta)(W^{*})\right)$
for every $\om\in\Lone{\G}$. This gives \prettyref{eq:antipode_W_extended}.
\end{proof}
\begin{rem}
Essentially the same proof, replacing $\Lone{\G}$ by $\LoneSharp{\G}$
and using \citep[Proposition A.1]{Brannan_Daws_Samei__cb_rep_of_conv_alg_of_LCQGs},
shows that for an arbitrary LCQG $\G$ (not necessarily with trivial
scaling group), \prettyref{eq:antipode_W} holds for all $\rho\in\Cz{\hat{\G}}^{*}$.
\end{rem}
The next result generalizes one direction of \citep[Corollary 3.9]{Izumi__non_comm_Poisson}
alluded to above.
\begin{prop}
\label{prop:triv_tau_tracial_state}Let $\G$ be a LCQG with trivial
scaling group. Then every tracial state $\rho$ of $\Cz{\hat{\G}}$
satisfies \prettyref{eq:amen_nuc}.
\end{prop}
\begin{proof}
Write $(\H_{\rho},\pi_{\rho},\xi_{\rho})$ for the GNS construction
of $\rho$. Since $\rho$ is a trace, the formula $J\pi_{\rho}(a)\xi_{\rho}:=\pi_{\rho}(a^{*})\xi_{\rho}$,
$a\in\Cz{\hat{\G}}$, defines an involutive anti-unitary $J$ on $\H_{\rho}$.
Let $\left(\eta_{i}\right)_{i\in\mathfrak{I}}$ be an orthonormal
basis of $\H_{\rho}$. Denote by $\mathcal{F}(\mathfrak{I})$ the
set of finite subsets of $\mathfrak{I}$ directed by inclusion. By
a standard argument \citep[Lemma A.5 and its proof]{Kustermans_Vaes__LCQG_C_star},
for every $X,Y\in\M(\Cz{\G}\tensormin\Cz{\hat{\G}})$, the net 
\begin{multline*}
\left(\sum_{i\in F}((\i\tensor\om_{\xi_{\rho},\eta_{i}})(\i\tensor\pi_{\rho})(Y))^{*}(\i\tensor\om_{\xi_{\rho},\eta_{i}})(\i\tensor\pi_{\rho})(X)\right)_{F\in\mathcal{F}(\mathfrak{I})}\\
=\left(\sum_{i\in F}((\i\tensor(\om_{\eta_{i},\xi_{\rho}}\circ\pi_{\rho}))(Y^{*}))(\i\tensor(\om_{\xi_{\rho},\eta_{i}}\circ\pi_{\rho}))(X)\right)_{F\in\mathcal{F}(\mathfrak{I})}
\end{multline*}
in $\M(\Cz{\G})\subseteq B(\Ltwo{\G})$ is bounded, and converges
strongly to $(\i\tensor\om_{\xi_{\rho}})(\i\tensor\pi_{\rho})(Y^{*}X)=(\i\tensor\rho)(Y^{*}X)$.
Let $x\in\Cz{\hat{\G}}$. Taking $Y:=(\one\tensor x^{*})W$ and $X:=W$,
we deduce that the bounded net 
\[
\left(\sum_{i\in F}((\i\tensor(\om_{\eta_{i},\xi_{\rho}}\circ\pi_{\rho}))(W^{*}(\one\tensor x)))(\i\tensor(\om_{\xi_{\rho},\eta_{i}}\circ\pi_{\rho}))(W)\right)_{F\in\mathcal{F}(\mathfrak{I})}
\]
converges strongly to $(\i\tensor\rho)(W^{*}(\one\tensor x)W)$. Thus,
for every $\om\in\Lone{\G}$,
\[
\lim_{F\in\mathcal{F}(\mathfrak{I})}(\om\circ\kappa)\left[\sum_{i\in F}((\i\tensor(\om_{\eta_{i},\xi_{\rho}}\circ\pi_{\rho}))(W^{*}(\one\tensor x)))(\i\tensor(\om_{\xi_{\rho},\eta_{i}}\circ\pi_{\rho}))(W)\right]=(\om\circ\kappa\circ(\i\tensor\rho))(W^{*}(\one\tensor x)W).
\]
By traciality of $\rho$, for every $b,c\in\Cz{\hat{\G}}$, we have
\[
(\om_{\xi_{\rho},\pi_{\rho}(c)\xi_{\rho}}\circ\pi_{\rho})(b)=\rho(c^{*}b)=\rho(bc^{*})=(\om_{J\pi_{\rho}(c)\xi_{\rho},\xi_{\rho}}\circ\pi_{\rho})(b).
\]
As $\pi_{\rho}(\Cz{\hat{\G}})\xi_{\rho}$ is dense in $\H_{\rho}$,
this entails that $\om_{\xi_{\rho},\eta}\circ\pi_{\rho}=\om_{J\eta,\xi_{\rho}}\circ\pi_{\rho}$
for all $\eta\in\H_{\rho}$. Consequently, by \prettyref{lem:antipode_W_extended},
for every $i\in\mathfrak{I}$,
\begin{multline*}
\kappa\left[((\i\tensor(\om_{\eta_{i},\xi_{\rho}}\circ\pi_{\rho}))(W^{*}(\one\tensor x)))(\i\tensor(\om_{\xi_{\rho},\eta_{i}}\circ\pi_{\rho}))(W)\right]\\
\begin{aligned} & =\kappa\left[(\i\tensor(\om_{\xi_{\rho},\eta_{i}}\circ\pi_{\rho}))(W)\right]\kappa\left[((\i\tensor(\om_{\eta_{i},\xi_{\rho}}\circ\pi_{\rho}))(W^{*}(\one\tensor x)))\right]\\
 & =(\i\tensor(\om_{\xi_{\rho},\eta_{i}}\circ\pi_{\rho}))(W^{*})((\i\tensor(\om_{\eta_{i},\xi_{\rho}}\circ\pi_{\rho}))(W(\one\tensor x)))\\
 & =(\i\tensor(\om_{J\eta_{i},\xi_{\rho}}\circ\pi_{\rho}))(W^{*})((\i\tensor(\om_{\xi_{\rho},J\eta_{i}}\circ\pi_{\rho}))(W(\one\tensor x))).
\end{aligned}
\end{multline*}
Since $\left(\eta_{i}\right)_{i\in\mathfrak{I}}$ is an orthonormal
basis of $\H_{\rho}$ and $J$ is anti-unitary, $\left(J\eta_{i}\right)_{i\in\mathfrak{I}}$
is also an orthonormal basis of $\H_{\rho}$. Using the analog of
the above reasoning with $\left(J\eta_{i}\right)_{i\in\mathfrak{I}}$
in lieu of $\left(\eta_{i}\right)_{i\in\mathfrak{I}}$ and taking
$Y:=W$ and $X:=W(\one\tensor x)$, the bounded net
\[
\left(\sum_{i\in F}(\i\tensor(\om_{J\eta_{i},\xi_{\rho}}\circ\pi_{\rho}))(W^{*})((\i\tensor(\om_{\xi_{\rho},J\eta_{i}}\circ\pi_{\rho}))(W(\one\tensor x)))\right)_{F\in\mathcal{F}(\mathfrak{I})}
\]
converges strongly to $(\i\tensor\rho)(W^{*}W(\one\tensor x))=\rho(x)\one$.
All in all, $(\om\circ\kappa\circ(\i\tensor\rho))(W^{*}(\one\tensor x)W)=\rho(x)\om(\one)$
for every $\om\in\Lone{\G}$, hence $(\kappa\circ(\i\tensor\rho))(W^{*}(\one\tensor x)W)=\rho(x)\one=\kappa(\rho(x)\one)$,
proving \prettyref{eq:amen_nuc} by the injectivity of $\kappa$.
\end{proof}
We give two proofs of implication \prettyref{enu:amen_nuc_1_b}$\implies$\prettyref{enu:amen_nuc_1_c}.
The first one uses \prettyref{thm:amen_weak_contain}. The second
is shorter and more direct, but it basically uses the same idea. 
\begin{proof}[Proof of \prettyref{thm:amen_nuc}]
\prettyref{enu:amen_nuc_1_a}$\implies$\prettyref{enu:amen_nuc_1_b}:
assume that $\hat{\G}$ is co-amenable. Then $\G$ is amenable, so
that $\Cz{\hat{\G}}$ is nuclear by \citep[Theorem 3.3]{Bedos_Tuset_2003}.
Let $\rho\in\Cz{\hat{\G}}$ be a state such that $(\i\tensor\rho)(W)=\one$.
By a multiplicative domains argument, $\rho$ is a character of $\Cz{\hat{\G}}$
\citep[Proof of Theorem 3.1]{Bedos_Tuset_2003}, from which \prettyref{eq:amen_nuc}
readily follows.

\prettyref{enu:amen_nuc_1_b}$\implies$\prettyref{enu:amen_nuc_1_c},
first proof: suppose that such $\rho$ exists. Writing $(\H_{\rho},\pi_{\rho},\xi_{\rho})$
for the GNS construction of $\rho$ and $W_{\rho}:=(\i\tensor\pi_{\rho})(W)$,
we get 
\[
(\i\tensor\om_{\xi_{\rho}})\left[W_{\rho}^{*}(\one\tensor y)W_{\rho}\right]=\om_{\xi_{\rho}}(y)\one
\]
for every $y\in\pi_{\rho}(\Cz{\hat{\G}})$, hence for every $y$ in
the von Neumann algebra $M:=\overline{\pi_{\rho}(\Cz{\hat{\G}})}^{weak}\subseteq B(\H_{\rho})$.
Since $\Cz{\hat{\G}}$ is nuclear, $M$ is injective by \citep[Theorem IV.2.2.13]{Blackadar__oper_alg_book}.
Let $E$ be a conditional expectation from $B(\H_{\rho})$ onto $M$.
Notice that $W_{\rho}\in\M(\Cz{\G}\tensormin\pi_{\rho}(\Cz{\hat{\G}}))\subseteq\Linfty{\G}\tensorn M$.
Precisely as in \citep[Theorem 2.4]{Soltan_Viselter__note_amenability_LCQGs}
(or the relevant part of the proof of \prettyref{thm:amen_weak_contain}),
for all $z\in B(\H_{\rho})$ and $\om\in\Lone{\G}$ we have
\[
E\left((\om\tensor\i)\left[W_{\rho}^{*}(\one\tensor z)W_{\rho}\right]\right)=(\om\tensor\i)\left[W_{\rho}^{*}(\one\tensor Ez)W_{\rho}\right],
\]
thus 
\[
(\om_{\xi_{\rho}}\circ E)\left((\om\tensor\i)\left[W_{\rho}^{*}(\one\tensor z)W_{\rho}\right]\right)=(\om\tensor\om_{\xi_{\rho}})\left[W_{\rho}^{*}(\one\tensor Ez)W_{\rho}\right]=(\om_{\xi_{\rho}}\circ E)(z)\om(\one).
\]
In conclusion, the co-representation $W_{\rho}$ is left amenable
with $\om_{\xi_{\rho}}\circ E$ a left-invariant mean. By definition,
$W_{\rho}$ is weakly contained in $W$, so from \prettyref{thm:amen_weak_contain}
we infer that $W$ is left amenable. This evidently implies that $\G$
is amenable (in fact, the converse is also true by \citep[Theorem 4.1]{Bedos_Tuset_2003}).

\prettyref{enu:amen_nuc_1_b}$\implies$\prettyref{enu:amen_nuc_1_c},
second proof: embed $\Cz{\hat{\G}}$ in $\Cz{\hat{\G}}^{**}$ canonically,
and view $\rho$ as a normal state of $\Cz{\hat{\G}}^{**}$. We will
regard $W\in\M(\Cz{\G}\tensormin\Cz{\hat{\G}})$ both as an element
of $\Cz{\G}^{**}\tensorn\Cz{\hat{\G}}^{**}$ and as an element of
$\Cz{\G}^{**}\tensorn B(\Ltwo{\G})$ depending on the context. By
a weak$^{*}$-continuity argument, we have 
\[
(\i\tensor\rho)\left(W^{*}(\one\tensor y)W\right)=\rho(y)\one\qquad(\forall y\in\Cz{\hat{\G}}^{**}).
\]
The assumption that $\Cz{\hat{\G}}$ is nuclear is equivalent to $\Cz{\hat{\G}}^{**}$
being an injective operator system \citep[Theorem IV.3.1.12]{Blackadar__oper_alg_book}.
Therefore, the embedding $\M(\Cz{\hat{\G}})\hookrightarrow\Cz{\hat{\G}}^{**}$
extends to a unital completely positive map $\Phi:B(\Ltwo{\G})\to\Cz{\hat{\G}}^{**}$.
Consider the unital completely positive map $\i\tensor\Phi:\Cz{\G}^{**}\tensorn B(\Ltwo{\G})\to\Cz{\G}^{**}\tensorn\Cz{\hat{\G}}^{**}$
given by \prettyref{lem:id_tensor_Phi}. For every $\om\in\Cz{\G}^{*}$,
\[
(\om\tensor\i)\left((\i\tensor\Phi)(W)\right)=\Phi\left((\om\tensor\i)(W)\right)=(\om\tensor\i)(W)
\]
because $(\om\tensor\i)(W)\in\M(\Cz{\hat{\G}})$. This means that
$(\i\tensor\Phi)(W)=W$. Thus $W$ belongs to the multiplicative domain
of $\i\tensor\Phi$, so that $(\i\tensor\Phi)\left(W^{*}XW\right)=W^{*}(\i\tensor\Phi)(X)W$
for every $X\in\Cz{\G}^{**}\tensorn B(\Ltwo{\G})$. Let $\theta:=\rho\circ\Phi$.
For every $x\in B(\Ltwo{\G})$, 
\[
(\i\tensor\theta)\left(W^{*}(\one\tensor x)W\right)=(\i\tensor\rho)\left(W^{*}(\one\tensor\Phi(x))W\right)=\theta(x)\one.
\]
Consequently, $\theta|_{\Linfty{\G}}$ is a left-invariant mean of
$\G$.

\prettyref{enu:amen_nuc_1_a}$\implies$\prettyref{enu:amen_nuc_2_trivial_tau}
for every LCQG $\G$ because every character is a tracial state. When
$\G$ has trivial scaling group, \prettyref{enu:amen_nuc_2_trivial_tau}$\implies$\prettyref{enu:amen_nuc_1_b}
by \prettyref{prop:triv_tau_tracial_state}.
\end{proof}
For discrete quantum groups, amenability is equivalent to co-amenability
of the dual, and we thus have the following consequence of \prettyref{thm:amen_nuc}.
\begin{cor}
Let $\G$ be a discrete quantum group. Then $\G$ is amenable if and
only if condition \prettyref{enu:amen_nuc_1_b} of \prettyref{thm:amen_nuc}
holds.
\end{cor}
We conjecture that condition \prettyref{enu:amen_nuc_1_b} is, in
fact, equivalent either to condition \prettyref{enu:amen_nuc_1_a}
or to condition \prettyref{enu:amen_nuc_1_c} for arbitrary LCQGs.
However, we were not able to verify this.
\begin{acknowledgement*}
The second-named author is indebted to Pawe{\l} Kasprzak and Adam
Skalski for intriguing conversations on the content of this paper
and related topics, which took place when he was visiting Adam Skalski
in Warsaw, and for their helpful remarks. The problem of extending
the equivalence between amenability of $G$ and nuclearity of $\Cr G$
for discrete groups $G$ to discrete quantum groups was suggested
to the second-named author several years ago by Piotr M.~So{\l}tan
(a similar problem was subsequently addressed in \citep{Soltan_Viselter__note_amenability_LCQGs}),
and for that he is grateful to him.
\end{acknowledgement*}
\bibliographystyle{amsalpha}
\bibliography{Amenability}

\end{document}